\documentclass[a4paper, 12pt]{article}

\usepackage[margin=1in]{geometry}
\usepackage{amsmath,amssymb,amsthm,mathtools}
\usepackage{microtype}
\usepackage{enumitem}
\usepackage[colorlinks=true,linkcolor=blue,citecolor=blue,urlcolor=blue]{hyperref}

\newtheorem{theorem}{Theorem}[section]
\newtheorem{proposition}[theorem]{Proposition}
\newtheorem{lemma}[theorem]{Lemma}
\newtheorem{corollary}[theorem]{Corollary}

\newtheorem{example}[theorem]{Example}
\newtheorem{problem}[theorem]{Problem}

\theoremstyle{definition}
\newtheorem{definition}[theorem]{Definition}

\DeclareMathOperator{\ord}{ord}
\DeclareMathOperator{\supp}{supp}

\title{The equality between the Erd\H{o}s-Ginzburg-Ziv constant and
the short product-one constant for finite nonabelian groups}
\author{
Yongke Qu
\qquad
Guoqing Wang\thanks{Corresponding author. E-mail: gqwang1979@aliyun.com.}
\qquad
Yuanlin Li}
\date{}

\begin{document}

\maketitle

\begin{abstract}

Let $G$ be a finite group, and let $\exp(G)$ denote its
exponent. The Erd\H{o}s-Ginzburg-Ziv constant $s(G)$ is the least
integer forcing a product-one subsequence of length $\exp(G)$, while
the short product-one constant $\eta(G)$ is the least integer forcing
a nonempty product-one subsequence of length at most $\exp(G)$. The
natural nonabelian extension of a conjecture [W. Gao, \emph{On zero-sum subsequences of restricted size II}, Discrete Math. 2003] on the Erd\H{o}s-Ginzburg-Ziv constant in finite abelian groups  predicts that
$s(G)=\eta(G)+\exp(G)-1.$
We confirm this equality for every finite nonabelian group $G$ having
a cyclic subgroup of index $p$, where $p$ is the smallest prime
divisor of $|G|$.  As further consequences, we determine all
generalized Erd\H{o}s-Ginzburg-Ziv constants $s_{m\exp(G)}(G)$ for this family of groups.
\end{abstract}


\section{Introduction}

The classical theorem of Erd\H{o}s, Ginzburg and Ziv \cite{EGZ}
asserts that every sequence of $2n-1$ elements of the cyclic group
$C_n$ contains an $n$-term zero-sum subsequence. This is one of the
foundational results of zero-sum theory. Its prescribed-length nature
has inspired the systematic study of zero-sum subsequences of
specified length over finite abelian groups; see
\cite{GaoGeroldingersurvey} for a survey.

For a finite group $G$ written multiplicatively, the natural analogue
of a zero-sum subsequence is a \emph{product-one subsequence}: its
terms can be ordered so that their product is the identity. Let
$$\exp(G)=\mathop{\rm lcm}\{\ord(g):g\in G\}$$
be the exponent of $G$. The Erd\H{o}s-Ginzburg-Ziv constant $s(G)$, referred to as the EGZ constant for short,
is the least integer $t$ such that every sequence over $G$ of length
at least $t$ contains a product-one subsequence of length exactly
$\exp(G)$. Its short-sequence counterpart is $\eta(G)$, the least
integer forcing a nonempty product-one subsequence of length at most
$\exp(G)$. A standard extremal construction gives
$s(G)\geq\eta(G)+\exp(G)-1.$
Gao \cite[Conjecture 2.3]{GaoRestrictedII} conjectured that this universal lower bound
is always attained for every finite abelian group $G$, namely,
$$
s(G)=\eta(G)+\exp(G)-1.
$$
For cyclic groups this equality is precisely the Erd\H{o}s-Ginzburg-Ziv
theorem. The conjecture remains open for general finite abelian
groups, but it is confirmed for every group of rank at most two. More
precisely, if $G\cong C_m\oplus C_n$ with $m\mid n$, then
$\eta(G)=2m+n-2$ and
$s(G)=2m+2n-3$ (see \cite[Theorem~6.3]{GaoGeroldingersurvey})
and therefore $s(G)=\eta(G)+\exp(G)-1$. The conjecture is also confirmed
for groups of exponent $2$, $3$, or $4$, and for $C_5^3$; see
\cite[Theorem 6.6]{GaoGeroldingersurvey} and the references therein.

In the nonabelian setting, the order of the selected terms is part of
the product-one condition, which makes the problem substantially more
delicate. Only a limited number of families have been treated.
Oh and Zhong determined $s(G)$ and solved the associated inverse problem for dihedral and dicyclic groups \cite{OhZhong2020}. Zhong also established the corresponding results for the short product-one constant \cite{Zhong2021}. Taken together, these results verify Gao's equality for both families of groups.
Avelar, Brochero
Mart\'{\i}nez and Ribas determined the relevant constants and confirmed
Gao's equality for a broad family of split metacyclic groups
$C_n\rtimes_s C_2$ \cite{AvelarBrocheroRibas2023}. Analogous results
for non-split metacyclic groups were subsequently obtained by Yang,
Zhang and Feng \cite{YangZhangFeng2024} and by Ribas
\cite{Ribas2025}.  Brochero Mart\'{\i}nez, Lemos, Moriya and Ribas
determined the four main product-one constants of
$D_{2n}\times C_2$, thereby verifying the equality for an important
family beyond the metacyclic setting \cite{BLMR}. Closely related prescribed-length product-one problems for a broader
class of metacyclic groups were studied by Han and Zhang \cite{HanZhang2019}. For
$G\cong C_m\ltimes_{\varphi}C_{mn}$, among other results, they determined the constant
$s_{mn\mathbb N}(G)$, and obtained upper bounds for
$s(G)$.

The preceding results concern important individual families, most of
which have a cyclic subgroup of index two. They naturally suggest
seeking a uniform structural theorem that is independent of a
particular presentation and is not restricted to index two. In this paper, we
consider finite nonabelian groups having a cyclic subgroup $H$ of
index $p$, where $p$ is the smallest prime divisor of the group order.
Such a subgroup is automatically normal.
This class contains, in
particular, all dihedral and dicyclic groups, as well as many other
metacyclic groups. Thus the theorem unifies and extends several
previously known cases. Our main result is the following.

\begin{theorem}\label{theorem:main}
Let $G$ be a finite nonabelian group, and let $p$ be the smallest prime
divisor of $|G|$. Suppose that $G$ has a cyclic subgroup $H$ of index $p$.
Then $$s(G)=\eta(G)+\exp(G)-1=\mathsf d(G)+\exp(G).$$ Moreover, the following assertions hold:
\begin{enumerate}[label=\rm(\roman*)]
\item If $\exp(G)=\frac{|G|}{p}$, then
$s(G)=2\frac{|G|}{p}+p-2.$
\item Otherwise, $\exp(G)=|G|$ and
$s(G)=|G|+\frac{|G|}{p}+p-2.$
\end{enumerate}
\end{theorem}

The theorem confirms the natural nonabelian extension of Gao's
conjecture for this class and gives
the stronger conclusion that the two lower bounds
$\eta(G)\geq\mathsf d(G)+1$
and
$s(G)\geq\eta(G)+\exp(G)-1$
are simultaneously attained.

The significance of $s(G)$ is also reflected in the generalized EGZ
constants. For $m\in\mathbb N$, let $s_{m\exp(G)}(G)$ be the least
integer forcing a product-one subsequence of length $m\exp(G)$. This
family contains both the ordinary EGZ constant and the Gao constant:
$$s_{\exp(G)}(G)=s(G)
\qquad\text{and}\qquad
s_{|G|}(G)=E(G),$$
where the second identity uses $\exp(G)\mid |G|$. It is also linked to
the small Davenport constant $\mathsf d(G)$ through the universal
lower bound
$$s_{m\exp(G)}(G)\geq \mathsf d(G)+m\exp(G).$$
Thus the generalized EGZ constants place $s(G)$, $E(G)$, and
$\mathsf d(G)$ in a single prescribed-length framework.

Combining Theorem \ref{theorem:main} with the general bounds developed
in Section 4, we further determine this entire family for the groups
under consideration. For every $m\geq 1$,
$$s_{m\exp(G)}(G)
 =\mathsf d(G)+m\exp(G)
 =\eta(G)+m\exp(G)-1.$$
Taking $m=|G|/\exp(G)$ recovers
$E(G)=\mathsf d(G)+|G|$.

The paper is organized as follows. Section 2 introduces the notation
and definitions used throughout. In Section 3, we establish the
necessary auxiliary results and prove the main theorem. Section 4
derives the prescribed-length consequences, investigates the
stabilization threshold associated with the generalized EGZ
constants, and formulates the corresponding problems for general
finite groups.

\section{Notation and terminology}

Throughout this paper, groups are finite and written multiplicatively.  The
identity element of a group $G$ is denoted by $1$, and its multiplication is
written as $*$.  If $a,b\in\mathbb R$ with $a\leq b$, then
$[a,b]=\{z\in\mathbb Z:a\leq z\leq b\}.$
For $A\subseteq G$, the subgroup generated by $A$ is denoted by
$\langle A\rangle$.  We write $C_n$ for a cyclic group of order $n$.

We use the standard language of sequences over groups; see
\cite{GG2013} and \cite[Chapter 5]{GHK}.  More precisely, a sequence over
$G$ is an element of the free abelian monoid $\mathcal F(G)$.  It may be
displayed either as
$$\mathop{\bullet}\limits_{g\in G} g^{[\mathsf v_g(S)]}
\qquad\text{or as}\qquad
S=g_1\boldsymbol{\cdot}\ldots\boldsymbol{\cdot}g_\ell,$$
where $\mathsf v_g(S)\in\mathbb N_0$ is the multiplicity of $g$ in $S$
and only finitely many multiplicities are nonzero. If $S_1, S_2 \in \mathcal F(G)$, then $S_1 \boldsymbol{\cdot} S_2 \in \mathcal F(G)$ denotes the sequence satisfying $$\mathsf v_g(S_1 \boldsymbol{\cdot} S_2) = \mathsf v_g(S_1) + \mathsf v_g( S_2)  \text{ for all } g \in G.$$
The symbol
$\boldsymbol{\cdot}$ distinguishes multiplication in $\mathcal F(G)$
from multiplication in $G$.  Brackets are used for powers in
$\mathcal F(G)$; thus
$g^{[k]}=\underbrace{g\boldsymbol{\cdot}\ldots
\boldsymbol{\cdot}g}_{k\ {\rm terms}}$
for $k\in\mathbb N_0$.
The length and support of $S$ are, respectively,
$|S|=\sum_{g\in G}\mathsf v_g(S)$ and $\supp(S)=\{g\in G:\mathsf v_g(S)>0\}.$
For $S,T\in\mathcal F(G)$, we call $T$ a subsequence of $S$ and  write $T\mid S$ if
$\mathsf v_g(T)\leq\mathsf v_g(S)$ for every $g\in G$.  In this case,
$S\boldsymbol{\cdot}T^{[-1]}$ denotes the sequence remaining after the
terms of $T$ have been deleted from $S$. Subsequences $T_1,\ldots,T_k$ of
$S$ are called mutually disjoint if
$T_1\boldsymbol{\cdot}\ldots\boldsymbol{\cdot}T_k\mid S.$ The empty sequence has length
zero.

A homomorphism $\varphi:G\to K$ acts termwise on sequences.  Namely, if
$S=g_1\boldsymbol{\cdot}\ldots\boldsymbol{\cdot}g_\ell$, then
$\varphi(S)=\varphi(g_1)\boldsymbol{\cdot}\ldots
\boldsymbol{\cdot}\varphi(g_\ell)\in\mathcal F(K).$
We use the same symbol $\varphi$ for the resulting monoid homomorphism
$\mathcal F(G)\to\mathcal F(K)$.

Let $S=g_1\boldsymbol{\cdot}\ldots\boldsymbol{\cdot}g_\ell$.  Because
$G$ need not be abelian, the product of all terms of $S$ may depend on
their order.  We therefore put
$$\pi(S)=
\bigl\{
g_{\tau(1)}*\cdots*g_{\tau(\ell)}:
\tau\ \text{is a permutation of }[1,\ell]
\bigr\}.$$
For $k\in[1,\ell]$, define
$\Pi_k(S)=\bigcup\limits_{\substack{T\mid S\\ |T|=k}}\pi(T)$ and
$\Pi(S)=\bigcup_{k=1}^{\ell}\Pi_k(S).$
Thus $\Pi_k(S)$ consists of all products obtained from $k$ terms of $S$,
whereas $\Pi(S)$ contains all products arising from nonempty
subsequences.  We call $S$ a \emph{product-one sequence} if
$1\in\pi(S)$, and \emph{product-one free} if $1\notin\Pi(S)$. A nonempty product-one sequence is called \emph{minimal} if it cannot be written as the product of two nonempty product-one sequences.
As usual,
the empty sequence is regarded as product-one.

\begin{samepage}
\begin{definition}\label{definition:invariants}
Let $G$ be a finite group. The exponent of $G$ is
$$\exp(G)=\operatorname{lcm}\{\ord(g):g\in G\}.$$  We denote by
\begin{itemize}
\item[$\bullet$] $\mathsf d(G)$ the largest integer $\ell$ such that
there exists a product-one free sequence over $G$ of length $\ell$;

\item[$\bullet$] $\mathsf D(G)$ the largest integer $\ell$ such that
there exists a minimal product-one sequence over $G$ of length $\ell$;

\item[$\bullet$] $\eta(G)$ the least integer $\ell$ such that every
sequence over $G$ of length at least $\ell$ contains a nonempty
product-one subsequence of length at most $\exp(G)$;

\item[$\bullet$] $s(G)$ the least integer $\ell$ such that every
sequence over $G$ of length at least $\ell$ contains a product-one
subsequence of length exactly $\exp(G)$;

\item[$\bullet$] $E(G)$ the least integer $\ell$ such that every
sequence over $G$ of length at least $\ell$ contains a product-one
subsequence of length exactly $|G|$.
\end{itemize}
\end{definition}
\end{samepage}

\section{Auxiliary results and proof of Theorem \ref{theorem:main}}

We begin with some structural lemmas for the groups considered in this paper.

\begin{lemma}\label{lemma:normality}\cite[Corollary 4.10]{Hungerford}
Let $p$ be the smallest prime divisor of $|G|$.  Every subgroup of
index $p$ in $G$ is normal.
\end{lemma}

A more general version of the following lemma can be found in \cite[pp. 129--130, Theorem 21]{Zassenhaus}. For the reader's convenience, we give a proof of it below.

\begin{lemma}\label{lemma:grouppresentation}
Suppose that $H=\langle y\rangle$ is cyclic of order $n$ and has index
$p$ in $G$.  If $H\trianglelefteq G$, then there is an element
$x\in G$ and integers $a,r$ such that
$$G=\langle x,y\mid y^n=1,\ x^p=y^a,\ x^{-1}yx=y^r\rangle,$$
where
$r^p\equiv1\pmod n$ and $a(r-1)\equiv0\pmod n.$
Moreover, $\exp(G)\in\{n,pn\}.$
\end{lemma}

\begin{proof}
Choose $x\in G$ so that $xH$ generates $G/H\cong C_p$.  Since
$x^p\in H$, we have $x^p=y^a$.  Normality of $H$ implies that
conjugation by $x$ induces an automorphism of $H$, so
$$x^{-1}yx=y^r$$ for some $r$ relatively prime to $n$.

Notice that conjugation by $x^p=y^a$ is trivial on $H$.  Applying conjugation by
$x$ successively $p$ times to $y$ therefore gives
$y^{r^p}=y$, and hence $r^p\equiv1\pmod n$.  Also $x$ commutes with
its power $x^p=y^a$.  Conjugating $y^a$ by $x$ gives
$y^{ar}=y^a$, and hence $a(r-1)\equiv0\pmod n$.

Since $|G|=pn$, it follows that $\exp(G)\mid pn$. On the other hand, since $\ord(y)=n$, it follows that $n\mid \exp(G)$, and so $\exp(G)\in \{n,pn\}$.
\end{proof}

\begin{lemma}\label{lemma:ranktwoquotient}
Let $G$ be a finite nonabelian group, and let $p$ be the smallest prime
divisor of $|G|$. Suppose that $G$ has a cyclic subgroup
$H=\langle y\rangle$
of index $p$ with $\exp(G)=|H|$.
Then $K=\langle y^p\rangle$
is normal in $G$ and $G/K\cong C_p\oplus C_p.$
\end{lemma}

\begin{proof} By Lemma \ref{lemma:normality}, we have $H\trianglelefteq G$.
Denote $n=|H|$. Since $\exp(G)=n$, we have $p\mid n$, for otherwise, the existence of the element of order $p$ contradicts $\exp(G)=n$. The subgroup $K=\langle y^p\rangle$ is characteristic in the cyclic
group $H$.  Since $H\trianglelefteq G$, characteristicity in $H$
implies $K\trianglelefteq G$. Since
$|H/K|=p$ and
$[G/K:H/K]=[G:H]=p$,
it follows that $|G/K|=p^2$. By Lemma \ref{lemma:grouppresentation}, we see that $G/K$ is generated by $\{xK,yK\}$. Notice $\ord(yK)=p$. To show the conclusion, it suffices to show that $\ord(xK)=p$. Since $\ord(x)\mid n$, we see
\begin{equation}\label{equation:ord(xpmid}
\ord(x^p)\mid \frac{n}{p}.
\end{equation} Since $x^p=y^a\in H$ and $|K|=\frac{n}{p}$, it follows from \eqref{equation:ord(xpmid} that $x^p$ belongs to the unique subgroup $K$ of order $\frac{n}{p}$ of $H$. This implies that $(xK)^p=x^p K=y^a K=K$, completing the proof.
\end{proof}

The following easy lemma will be used often when we lift a result from the quotient group to the original group.

\begin{lemma}\label{lemma:quotientlifting} Let $G$ be a finite group, and let $K$ be a normal subgroup of $G$ such that $G/K$ is abelian.  Let $T$ be a sequence over $G$, and let $\varphi$ be the canonical epimorphism of $G$ onto $G/K$. If $\varphi(T)$ is product-one over $G/K$, then
$\pi(T)\subseteq K.$
\end{lemma}

The following prescribed-length theorem is due to Gao; see
\cite[Theorem 3.2]{GaoRestrictedII} or
\cite[Theorem 6.10(2)]{GaoGeroldingersurvey}.

\begin{lemma}\label{lemma:gao>D+G-1}
Let $G$ be a finite abelian group,
and let $N$ be a positive integer
such that $\exp(G)\mid N$ and $N\geq |G|$.  Every sequence over $G$ of
length at least $N+\mathsf d(G)$
has a product-one subsequence of length $N$.
\end{lemma}

We use the following results for finite abelian groups of rank at most two.

\begin{lemma}\label{lemma:someresultsrank2}
Let $k\geq 2$ be an integer. Then the following conclusions hold.

\begin{enumerate}[label=\rm(\roman*)]
\item Every sequence of length $4k-3$ over $C_k\oplus C_k$ has a
product-one subsequence of length $k$ (see \cite{Reiher}, or \cite[Theorem 4.2.10]{GRuzsa});

\item Every sequence of length $3k-2$ over $C_k\oplus C_k$ has a
product-one subsequence of length $k$ or $2k$ (\cite[Theorem 6.7 (2)]{GaoGeroldingersurvey});

\item Every sequence of length $2k-1$ over $C_k$ has a product-one
subsequence of length $k$ (see \cite{EGZ});

\item If a sequence $T$ over $C_k$ has length $2k-2$ and has no product-one subsequence of length $k$, then
$$T=\alpha^{[k-1]}\boldsymbol{\cdot}\beta^{[k-1]},$$
where $\alpha\beta^{-1}$ generates $C_k$ (see \cite[Lemma 4]{BialostockiDierker} and \cite[Theorem 3.1]{FloresOrdaz}).
\end{enumerate}
\end{lemma}

We next establish a lemma concerning the special factorizations used in the proof of our main theorem.

\begin{lemma}\label{lemma:pairfactorization}
Let $W$ be a product-one sequence of length $4p$ over
$C_p\oplus C_p$, and let $u,v$ be two given terms of
$W$.  Then $W$ has a factorization
$W=W_1\boldsymbol{\cdot}W_2\boldsymbol{\cdot}W_3$
into product-one subsequences satisfying
$|W_1|=|W_2|=p$, $|W_3|=2p,$
such that $u$ and $v$ occur in the same product-one subsequence $W_i$ for some $i\in [1,3]$.
\end{lemma}

\begin{proof} Remove the specified occurrences $u,v$.  The remaining sequence has
length $4p-2$, so Lemma \ref{lemma:someresultsrank2} (i) gives a $p$-term
product-one subsequence $W_1$ disjoint from $u,v$.  Put
$$R=W\boldsymbol{\cdot}W_1^{[-1]}.$$
Then $R$ is product-one, has length $3p$, and contains $u,v$.

Apply Lemma \ref{lemma:someresultsrank2} (ii) to
$R\boldsymbol{\cdot}(u\boldsymbol{\cdot}v)^{[-1]}$, whose length is
$3p-2$.  It has a product-one subsequence $Y$ of length $p$ or $2p$.
Since $R$ is product-one, the complementary sequence
$R\boldsymbol{\cdot}Y^{[-1]}$ is also product-one.  If $|Y|=p$, take
$W_2=Y$ and $W_3=R\boldsymbol{\cdot}Y^{[-1]}$.  If $|Y|=2p$, reverse
these choices.  In both cases $u,v$ lie in the same product-one subsequence,
and the required lengths are $p,p,2p$, as desired.
\end{proof}

The following lemma is a key ingredient in the proof of our main theorem.

\begin{lemma}\label{lemma:blockwithstructure} Let $G$ be a finite group, and let $K$ be a cyclic normal subgroup of $G$ of order $|K|=q\geq 2$ such that
$G/K\cong C_p\oplus C_p$ for some prime $p$.  Let
$T=V_1\boldsymbol{\cdot}\ldots\boldsymbol{\cdot}V_{2q-2}
\boldsymbol{\cdot}U \in \mathcal{F}(G)$
be such that
$|V_1|=\cdots=|V_{2q-2}|=p$, $|U|=2p,$
and all $\varphi(V_1),\ldots,\varphi(V_{2q-2}),\varphi(U)$ are product-one sequences over
$G/K$, where $\varphi: G\rightarrow G/K$ is the canonical epimorphism. Suppose
that $T$ has no product-one subsequence of length $pq$.  Then the following conclusions hold.

\begin{enumerate}[label=\rm(\roman*)]
\item $|\pi(V_1)|=\cdots=|\pi(V_{2q-2})|=|\pi(U)|=1$, and $T$ is a product-one sequence;
\item the subgroup generated by $\supp(T)$ is abelian.
\end{enumerate}
\end{lemma}

\begin{proof}
(i) By Lemma \ref{lemma:quotientlifting}, $$\pi(V_1),\ldots,\pi(V_{2q-2}), \pi(U)\subseteq K.$$
Choose $\alpha_i\in\pi(V_i)$ for $1\leq i\leq2q-2$.  The sequence
$A=\alpha_1\boldsymbol{\cdot}\ldots
\boldsymbol{\cdot}\alpha_{2q-2}\in \mathcal{F}(K)$
has no $q$-term product-one subsequence.  Otherwise, by
ordering and concatenating the corresponding $q$ subsequences, we would
obtain a product-one subsequence of $T$ of length $pq$.  By Lemma \ref{lemma:someresultsrank2} (iv), we have
\begin{equation}\label{equation A=alphabata}
A=\alpha^{[q-1]}\boldsymbol{\cdot}\beta^{[q-1]},
\end{equation}
where $\alpha,\beta\in K$ and $\alpha\beta^{-1}$ generates the cyclic group $K$.

We first prove that each $\pi(V_i)$ is a singleton.  Replace $\alpha_i$ by an arbitrary other element of
$\pi(V_i)$, and we obtain a sequence $A'$.  The resulting sequence $A'$ of $2q-2$ terms of $K$ still
has no $q$-term product-one subsequence, so it must again have two distinct values with each occurring $q-1$ times in the replaced sequences by Lemma \ref{lemma:someresultsrank2} (iv).

Suppose first that $q\geq 3$.  If the changed term was originally
$\alpha$, the unchanged terms include $q-2$ copies of $\alpha$ and
$q-1$ copies of $\beta$.  A sequence in which exactly two values each
occur $q-1$ times can then be obtained only if the changed value is
$\alpha$.  The same argument applies to a term originally equal to
$\beta$.  If $q=2$, the sequence $A$ consists of exactly the two distinct elements of $C_2$ as all its terms;  once either of the two is fixed, the other is uniquely determined.  Thus $\pi(V_i)$ is a singleton for every $i$.

Now take an arbitrary element $\gamma\in\pi(U)$.  We claim that
\begin{equation}\label{equation gamma*betaneq1}
\gamma*\alpha^j*\beta^{q-2-j}\neq 1 \text{ for each } j\in [0,q-2].
\end{equation}
Since otherwise, the sequence $U$ together with $j$ sequences corresponding to $\alpha$ and
$q-2-j$ sequences corresponding to $\beta$ would be a
product-one sequence of length $2p+(q-2)p=pq,$
a contradiction.

Notice that the $q$ elements $\alpha^{j}*\beta^{q-2-j}=(\alpha*\beta^{-1})^j*\beta^{(q-2)}$ are all distinct as $j$ ranges over $[0,q-1]$. Since $\pi(U)\subseteq K$, it follows from \eqref{equation gamma*betaneq1} that $\gamma\in K\setminus \bigl\{(\alpha^{j}*\beta^{q-2-j})^{-1}:0\leq j\leq q-2\bigr\}=\bigl\{(\alpha^{j}*\beta^{q-2-j})^{-1}:0\leq j\leq q-1\bigr\}\setminus \bigl\{(\alpha^{j}*\beta^{q-2-j})^{-1}:0\leq j\leq q-2\bigr\}=\{(\alpha^{q-1}*\beta^{q-2-(q-1)})^{-1}\}=\{\alpha*\beta\}$. By the arbitrariness of choosing $\gamma$ from $\pi(U)$, we conclude that
\begin{equation}\label{equation pi(U)=}
\pi(U)=\{\alpha*\beta\}
\end{equation}
 and $$|\pi(U)|=1.$$
Combined \eqref{equation A=alphabata} and \eqref{equation pi(U)=}, we derive that $1=\alpha^{q-1}*\beta^{q-1}*(\alpha*\beta)\in \pi(T)$.
This proves (i).

(ii) If $|\supp(T)|=1$, the conclusion is trivial. Hence, we assume $|\supp(T)|>1$, and take any two distinct elements $g,h\in \supp(T)$. It suffices to show that $g*h=h*g$. Since $2q-2\geq 2$, we can take two distinct integers $s,t\in [1,2q-2]$ such that $g\boldsymbol{\cdot} h\mid V_s\boldsymbol{\cdot} V_t\boldsymbol{\cdot} U$ (note that $g\boldsymbol{\cdot} h$ may belong to the same sequence of the three, this will not affect the following argument).
Since $\varphi(V_s\boldsymbol{\cdot} V_t\boldsymbol{\cdot} U)$ is a product-one sequence over $G/K$ of length $4p$, it follows from
Lemma \ref{lemma:pairfactorization} that $V_s\boldsymbol{\cdot} V_t\boldsymbol{\cdot} U=V_s'\boldsymbol{\cdot} V_t'\boldsymbol{\cdot} U'$ such that $\varphi(V_s'), \varphi(V_t'), \varphi(U')$ are product-one sequences over $G/K$ of lengths $p, p, 2p$ respectively, and that
the two terms $g,h$ belong to the same sequence of the three $V_s', V_t', U'$. Then we have a new factorization
$$T=V_1\boldsymbol{\cdot}\ldots\boldsymbol{\cdot}V_{2q-2}
\boldsymbol{\cdot}U\boldsymbol{\cdot} (V_s\boldsymbol{\cdot} V_t\boldsymbol{\cdot} U)^{[-1]}\boldsymbol{\cdot} (V_s'\boldsymbol{\cdot} V_t'\boldsymbol{\cdot} U'),$$ which satisfies the hypotheses of the lemma.  By Conclusion (i), the product set of the subsequence containing $g$ and $h$ is a
singleton. Arrange its terms so that $g$ and $h$ are adjacent. Exchanging
these two terms does not change the product, which gives $g*h=h*g$. This proves (ii).
\end{proof}

To prove the main theorem, we need three further lemmas. The first
inequality in Lemma \ref{lemma:universallower} is the standard
extremal lower bound for $s(G)$. For completeness, we include its
short proof.

\begin{lemma}\label{lemma:universallower}
For every finite group $G$, we have $$s(G)\geq\eta(G)+\exp(G)-1\geq \mathsf d(G)+\exp(G).$$
\end{lemma}

\begin{proof}
Choose a sequence $V\in \mathcal{F}(G)$ of length $\eta(G)-1$ having no nonempty
product-one subsequence of length at most $\exp(G)$, and set
$T=V\boldsymbol{\cdot}1^{[\exp(G)-1]}.$ It is easy to check that $T$ contains no product-one subsequence of length $\exp(G)$, which gives that $s(G)\geq\eta(G)+\exp(G)-1$.
The second inequality follows from $\eta(G)\geq \mathsf d(G)+1$.
\end{proof}

\begin{lemma}\cite{QuGaoLi}\label{lemma:E(G)=d+|G|} Let $G$ be a finite group that has a cyclic subgroup of index $p$, where $p$ is the smallest prime divisor of $|G|$. Then $E(G)=\mathsf d(G)+|G|.$
\end{lemma}

\begin{lemma} (\cite{QuLiTeeuwsenJCTA} and \cite[Corollary 1.3]{QLT}) \label{Lemma:smalld(G)} Let $G$ be a finite non-cyclic group and $p$ the smallest prime divisor of $|G|$. Then $\mathsf d(G)\leq \frac{|G|}{p}+p-2$, and furthermore, the equality  $\mathsf d(G)=\frac{|G|}{p}+p-2$ holds if and only if
$G$ has a cyclic subgroup of index $p$.
\end{lemma}

Now we are in a position to show Theorem \ref{theorem:main}.

\bigskip

\noindent{\bf Proof of Theorem \ref{theorem:main}.}  By Lemma \ref{lemma:normality}, $H$ is normal in $G$. Denote $|H|=n$. By Lemma~\ref{lemma:grouppresentation}, $\exp(G)\in \{n,pn\}$. Hence,
we shall distinguish two cases according to the value of $\exp(G)$.

\medskip

\noindent \textbf{Case 1. $\exp(G)=n$.}

Since $G$ always has an element of order $p$, it follows from the definition that $$p\mid n.$$
Choose the presentation for $G$
by Lemma \ref{lemma:grouppresentation}, put $q=n/p$, and let
$K=\langle y^p\rangle$.  Since $G$ is nonabelian, we have $q\geq 2$, otherwise $|G|=p^2$, which
would imply that $G$ is abelian. Lemma \ref{lemma:ranktwoquotient} gives
$C_q\cong K\unlhd G$ and
$G/K\cong C_p\oplus C_p.$ Let $\varphi:G\to G/K$ be the canonical epimorphism.
Notice that
$W=y^{[n-1]}\boldsymbol{\cdot}x^{[p-1]}$
is a product-one free sequence over $G$. This gives
\begin{equation}\label{equation eta(G)geq}
\eta(G)\geq \mathsf d(G)+1\geq |W|+1=n+p-1.
\end{equation}
We shall prove
\begin{equation}\label{equuation proves(G)=}
s(G)\leq2n+p-2.
\end{equation}
Assume to the contrary that there exists a sequence $S\in \mathcal{F}(G)$ of length $2n+p-2$ such that $S$ has no product-one subsequence of length $n$.
Let $V_1,\ldots,V_r$ be mutually disjoint subsequences of $S$ with $r$ being {\bf maximal}, such that $|V_1|=\cdots=|V_r|=p$, and $\varphi(V_1),\ldots,\varphi(V_r)$ are product-one over $G/K$. Denote $R=S\boldsymbol{\cdot} (V_1\boldsymbol{\cdot}\ldots\boldsymbol{\cdot} V_r)^{[-1]}.$
By the maximality of $r$, we conclude that $\varphi(R)$ contains no product-one subsequence of length $p$.
Lemma \ref{lemma:someresultsrank2} (i) therefore gives
$|R|\leq4p-4.$
This implies that $$r=\frac{|S|-|R|}{p}\geq  \left\lceil\frac{(2n+p-2)-(4p-4)}{p}\right\rceil=2q-2.$$
Fix an element $\alpha_i\in\pi(V_i)\subseteq K$ for each $i\in [1,r]$.  If
$r\geq2q-1$, Lemma \ref{lemma:someresultsrank2} (iii) gives $q$ of
the $\alpha_i$ whose product is $1$.  Ordering and concatenating the
corresponding subsequences gives a product-one subsequence of $S$ of length
$pq=n$, a contradiction.  Hence,
$r=2q-2.$
It follows that
$|R|=|S|-|V_1\boldsymbol{\cdot} \ldots\boldsymbol{\cdot} V_{r}|=(2pq+p-2)-(2q-2)p=3p-2.$
By Lemma \ref{lemma:someresultsrank2} (ii) and the maximality of $r$, we derive that $R$ contains a subsequence $U$ of length $2p$ such that $\varphi(U)$ is product-one over $G/K$.
Let
$$T=V_1\boldsymbol{\cdot}\ldots
\boldsymbol{\cdot}V_{2q-2}\boldsymbol{\cdot}U.$$
Then $|T|=2pq=2n$.  The sequence $T$ has no $n$-term product-one
subsequence because $T\mid S$.  Lemma \ref{lemma:blockwithstructure} shows
that $T$ is product-one and that $A=\langle\supp(T)\rangle$
is abelian. Since $G$ is nonabelian, and $p$ is the smallest prime divisor of $|G|$, it follows that $|A|\leq \frac{|G|}{p}=n$, and
$\exp(A)\mid\exp(G)=n$.  Apply Lemma \ref{lemma:gao>D+G-1} with
$N=n$.  Since $|T|=2n\geq n+|A|\geq  n+\mathsf d(A)$, it follows that $T$ has an $n$-term product-one subsequence. This
contradicts the choice of $S$, and so proves \eqref{equuation proves(G)=}.

By \eqref{equation eta(G)geq} and Lemma \ref{lemma:universallower}, we see that $s(G)\geq\eta(G)+n-1\geq2n+p-2.$
Combined with \eqref{equuation proves(G)=}, we conclude that $s(G)=\eta(G)+n-1=2n+p-2$ and $\eta(G)=\mathsf d(G)+1$.

\medskip

\noindent \textbf{Case 2. $\exp(G)=pn$.}

By definition, we derive that
$s(G)=E(G)$ and $\eta(G)\geq \mathsf d(G)+1$. It follows from Lemma \ref{lemma:E(G)=d+|G|} that $s(G)=E(G)=\mathsf d(G)+|G|\leq (\eta(G)-1)+\exp(G)$. Combined with Lemma \ref{lemma:universallower}, we have $s(G)=\eta(G)+\exp(G)-1$, and moreover, $\eta(G)-1=\mathsf d(G)$. Combined with Lemma \ref{Lemma:smalld(G)}, we derive that
$s(G)=|G|+\frac{|G|}{p}+p-2$. \qed

\section{Further consequences and problems}

We first remark that the equality $s(G)=\mathsf d(G)+\exp(G)$ in Theorem \ref{theorem:main} does not hold for all finite groups. For example, take a finite abelian group $G=C_n\oplus C_n$. It is known that $s(G)=4n-3=(3n-2)+n-1=\eta(G)+\exp(G)-1>(2n-2)+n=\mathsf d(G)+\exp(G)$. In the following, we shall employ a nonabelian finite group $G$ investigated in  \cite{BLMR}, for which the equality $s(G)=\eta(G)+\exp(G)-1$ holds, but $s(G)>\mathsf d(G)+\exp(G)$.

\begin{example} Let $n\geq 4$ be an even number, and let $G=D_{2n}\times C_2$ where $D_{2n}$ is the dihedral group of order $2n$. By Theorem 3.2 and Theorem 4.2 in \cite{BLMR}, we see $\exp(G)=n$, $\mathsf d(G)=n+1$, $\eta(G)=n+3$ and $s(G)=2n+2$, and therefore, $s(G)=\eta(G)+\exp(G)-1>\mathsf d(G)+\exp(G)$.
\end{example}

Combined with Theorem \ref{theorem:main}, the following problem arises naturally.

\begin{problem}\label{problem1}  Determine all finite groups such that $s(G)=\mathsf d(G)+\exp(G)$.
\end{problem}

In the following, we shall notice that for a {\bf generalized} EGZ invariant (including the Gao constant $E(G)$), the key is the relation with the small Davenport constant $\mathsf d(G)$ instead of the relation with short product-one invariant  $\eta(G)$.
To learn the mechanism underlying this, we need to introduce the generalized EGZ constant and the stabilization threshold (an associated existence invariant) in the setting of finite groups. Although these invariants were originally defined for finite
abelian groups in \cite[Definition 2.1]{GaoGeroldingersurvey} and \cite[Definition 3.5]{GaoRestrictedII}, their
definitions extend without change to arbitrary finite groups when zero-sum
subsequences are replaced by product-one subsequences.

\begin{definition} Let $G$ be a finite group and let $m$ be a positive integer. We denote
by $s_{m\exp(G)}(G)$ the least positive integer $t$ such that every
sequence over $G$ of length at least $t$ contains a product-one
subsequence of length exactly $m\exp(G)$. In particular,
$$s_{\exp(G)}(G)=s(G).$$
We also denote by $s_{\exp(G)\mathbb N}(G)$ the least positive integer
$t$ such that every sequence over $G$ of length at least $t$ contains
a nonempty product-one subsequence whose length is divisible by
$\exp(G)$.

The \emph{stabilization threshold} of $G$ is
$$\ell(G)=\inf\bigl\{t\in\mathbb N:
s_{m\exp(G)}(G)=\mathsf d(G)+m\exp(G)
\text{ for every }m\geq t\bigr\},$$
where the infimum of the empty set is understood to be infinity.
\end{definition}

The invariant $\ell(G)$ for finite abelian groups $G$ already appeared in
\cite[Definition 3.5]{GaoRestrictedII}; it was subsequently studied
systematically by Gao and Thangadurai \cite{GaoThangadurai}.
For a finite abelian group $G$, Gao's prescribed-length zero-sum
theorem, stated as Lemma \ref{lemma:gao>D+G-1}, can be reformulated as
$$\ell(G)\leq\frac{|G|}{\exp(G)}.$$
In fact, $\ell(G)$ provides a natural framework for studying the
relationships among the generalized EGZ constants $s_{m\exp(G)}(G)$, the Gao constant $E(G)$,
the small Davenport constant $\mathsf d(G)$, and the exponent $\exp(G)$ of $G$. The following
proposition makes these relationships explicit.

\begin{proposition}\label{proposition:ifsg=d+exp}
Let $G$ be a finite group.  Then $$0\leq E(G)-\bigl(\mathsf d(G)+|G|\bigr)
\leq s(G)-\bigl(\mathsf d(G)+\exp(G)\bigr),$$
and $$s(G)=\mathsf d(G)+\exp(G) \quad \text{if and only if} \quad \ell(G)=1.$$ Moreover, if $\ell(G)=1$ then the following assertions hold.
\begin{enumerate}[label=\rm(\roman*)]
\item For every positive integer $m$,
$s_{m\exp(G)}(G)=\eta(G)+m\exp(G)-1;$
\item $s_{\exp(G)\mathbb N}(G)
=s(G)
=\eta(G)+\exp(G)-1.$
\end{enumerate}
\end{proposition}

\begin{proof}
Let $U\in \mathcal F(G)$ be a product-one free sequence of length $\mathsf d(G)$.

We first prove that, for every
positive integer $m$,
\begin{equation}\label{equation:prescribedbounds}
\mathsf d(G)+m\exp(G)\leq s_{m\exp(G)}(G)\leq s(G)+(m-1)\exp(G).
\end{equation}
 Notice that the
sequence
$U\boldsymbol{\cdot}1^{[m\exp(G)-1]}$
has no product-one subsequence of length $m\exp(G)$. This implies
$s_{m\exp(G)}(G)\geq |U\boldsymbol{\cdot}1^{[m\exp(G)-1]}|+1=\mathsf d(G)+m\exp(G)$.
For the second inequality, let $S$ be a sequence of length
$s(G)+(m-1)\exp(G)$. Applying the definition of $s(G)$ successively, we obtain $m$
mutually disjoint product-one subsequences $T_1,\ldots,T_m$ of $S$, each of length $\exp(G)$.
Then $T_1\boldsymbol{\cdot}\ldots\boldsymbol{\cdot} T_m$ is a product-one subsequence of $S$ of length $m\exp(G)$. This proves \eqref{equation:prescribedbounds}.

Taking
$m=\frac{|G|}{\exp(G)}$ in \eqref{equation:prescribedbounds} and using
$E(G)=s_{m\exp(G)}(G)$, we obtain
$\mathsf d(G)+|G|\leq E(G)\leq s(G)+|G|-\exp(G).$
Subtracting $\mathsf d(G)+|G|$ yields
$$0\leq E(G)-\bigl(\mathsf d(G)+|G|\bigr)
\leq s(G)-(\mathsf d(G)+\exp(G)).$$

We now prove that
$$s(G)=\mathsf d(G)+\exp(G)
\quad \text{if and only if}  \quad
\ell(G)=1.$$  If $\ell(G)=1$, then taking $m=1$ in the definition immediately gives
$s(G)=\mathsf d(G)+\exp(G)$.
Hence, we next assume
$s(G)=\mathsf d(G)+\exp(G)$  and prove $\ell(G)=1$. By \eqref{equation:prescribedbounds}, we derive that for every positive integer $m$, $\mathsf d(G)+m\exp(G)\leq s_{m\exp(G)}(G)\leq s(G)+(m-1)\exp(G)=(\mathsf d(G)+\exp(G))+(m-1)\exp(G)=\mathsf d(G)+m\exp(G)$ and so $s_{m\exp(G)}(G)=\mathsf d(G)+m\exp(G)$. Therefore, $\ell(G)=1$, as required.

Suppose now that $\ell(G)=1$, i.e.,
$s(G)=\mathsf d(G)+\exp(G)$. By
Lemma \ref{lemma:universallower}, we have
\begin{equation}\label{equation:eta=d+1}
\eta(G)=\mathsf d(G)+1
\qquad\text{and}\qquad
s(G)=\eta(G)+\exp(G)-1.
\end{equation}
Combined with the definition of $\ell(G)$,  we have that $s_{m\exp(G)}(G)=\mathsf d(G)+m \exp(G)=\eta(G)+m \exp(G)-1$ for every $m\geq 1$,
which proves \rm(i).

By definition,
$s_{\exp(G)\mathbb N}(G)\leq s(G)$.
On the other hand, the sequence $U\boldsymbol{\cdot}1^{[\exp(G)-1]}$
contains no product-one subsequence whose length is a positive
multiple of $\exp(G)$. Consequently,
$s_{\exp(G)\mathbb N}(G)\geq |U\boldsymbol{\cdot}1^{[\exp(G)-1]}|+1=\mathsf d(G)+\exp(G)=s(G).$
Together with \eqref{equation:eta=d+1}, this gives
$s_{\exp(G)\mathbb N}(G)
=s(G)
=\eta(G)+\exp(G)-1,$
and proves \rm(ii).
\end{proof}

By Proposition \ref{proposition:ifsg=d+exp} and Theorem \ref{theorem:main}, we can derive the following immediately.

\begin{corollary}\label{corollary:generalized}
Let $G$ be a finite nonabelian group, and let $p$ be the smallest prime
divisor of $|G|$. Suppose that $G$ has a cyclic subgroup $H$ of index $p$.
Then $s_{m\exp(G)}(G)=\eta(G)+m\exp(G)-1$ for every positive integer $m$,
and $s_{\exp(G)\mathbb N}(G)=s(G)=\eta(G)+\exp(G)-1=\mathsf d(G)+\exp(G).$ In particular,
$E(G)=\mathsf d(G)+|G|.$
\end{corollary}

Notice that Corollary \ref{corollary:generalized} does not provide an
independent proof of Lemma \ref{lemma:E(G)=d+|G|}, since Lemma \ref{lemma:E(G)=d+|G|} is used in Case 2 of
the proof of Theorem \ref{theorem:main}.

By Proposition \ref{proposition:ifsg=d+exp}, Problem \ref{problem1}
is equivalent to characterizing the finite groups $G$ for which
$\ell(G)=1$. For finite abelian groups, this occurs precisely when
$G$ is cyclic (as also noted in \cite{GaoHanPengSun}). The nonabelian situation is considerably richer:
Theorem \ref{theorem:main} shows that $\ell(G)=1$ whenever $G$ has a
cyclic subgroup whose index is the smallest prime divisor of $|G|$.

Moreover, Gao et al. \cite[Conjecture 4.7]{GaoHanPengSun} also conjectured
that
$$
\ell(G)=\left\lceil\frac{\mathsf D(G)}{\exp(G)}\right\rceil
$$
for every finite abelian group $G$. The conjecture remains open in general. The following example shows that the proposed
formula does not extend verbatim to finite nonabelian groups.

\begin{example}\label{example:nonabelianell}
Let $n\geq4$ be even and let $G=D_{2n}$. By applying Corollary \ref{corollary:generalized} with $p=2$,  we derive that $s(G)=\mathsf d(G)+\exp(G)$, i.e.,
$\ell(G)=1.$ On the other hand, we see that
$\exp(G)=n$, $\mathsf d(G)=n.$ Moreover, by \cite[Theorem 1.1]{GG2013}, we have
$\mathsf D(G)=\mathsf d(G)+|G'|=n+\frac{n}{2}=\frac{3n}{2}.$ Therefore,
$$\left\lceil\frac{\mathsf d(G)+1}{\exp(G)}\right\rceil=\left\lceil\frac{\mathsf D(G)}{\exp(G)}\right\rceil
=2\neq\ell(G).$$
\end{example}

The preceding discussion leads to the following problem
concerning the stabilization threshold of a finite group.

\begin{problem}
Is $\ell(G)<\infty$ for every finite group $G$? Furthermore, establish
general upper bounds for $\ell(G)$, and determine its exact value for
natural families of finite nonabelian groups, whenever $\ell(G)$ is finite.
\end{problem}

\bigskip

\noindent {\bf Acknowledgments}

\noindent This work was supported in part by the National Natural Science Foundation of China (NSFC) (No. 12371335), by the Henan Provincial Selective Research Funding Program for Returned Scholars Studying Abroad (No. HNLX202611), and it was also supported in part by a Discovery Grant from the Natural Sciences and Engineering Research Council of Canada (No. RGPIN 2024-03931).

\section*{Author information}

\noindent\textbf{Yongke Qu}\\
Department of Mathematics, Luoyang Normal University,
Luoyang 471934, P.R. China\\
E-mail: \href{mailto:yongke1239@163.com}{yongke1239@163.com}

\medskip

\noindent\textbf{Guoqing Wang}\\
School of Mathematical Sciences, Tiangong University,
Tianjin 300387, P.R. China\\
E-mail: \href{mailto:gqwang1979@aliyun.com}{gqwang1979@aliyun.com}

\medskip

\noindent\textbf{Yuanlin Li}\\
Department of Mathematics and Statistics, Brock University,
St. Catharines, ON L2S 3A1, Canada\\
E-mail: \href{mailto:yli@brocku.ca}{yli@brocku.ca}

\end{document}